\documentclass[12pt, leqno]{amsart}

\usepackage[OT2,T1]{fontenc}
\DeclareSymbolFont{cyrletters}{OT2}{wncyr}{m}{n}
\DeclareMathSymbol{\Sha}{\mathalpha}{cyrletters}{"58}

\usepackage{indentfirst}
\usepackage{amstext}
\usepackage{amsopn}
\usepackage{amsfonts}
\usepackage{amsmath}
\usepackage{latexsym}
\usepackage{amscd}
\usepackage{amssymb}
\usepackage{amsmath}
\usepackage[all,cmtip]{xy}
\usepackage{leftidx}
\usepackage{graphicx}
\usepackage{tikz}

=\oddsidemargin \font\teneufm=eufm10 \font\seveneufm=eufm7
\font\fiveeufm=eufm5
\newfam\eufmfam
\textfont\eufmfam=\teneufm \scriptfont\eufmfam=\seveneufm
\scriptscriptfont\eufmfam=\fiveeufm

\let\goth\mathfrak

\def\cT{\mathcal T}

\def\GG{\mathbb{G}}

\def\gG{\goth G}

\def\gX{\goth X}

\def\1{\mbox{\bf 1}}

\DeclareMathOperator{\Aut}{Aut}

\DeclareMathOperator{\Out}{Out}

\DeclareMathOperator{\SO}{\rm SO}

\DeclareMathOperator{\cd}{cd}

\newcommand{\calB}{{\mathcal B}} 
\newcommand{\calU}{{\mathcal U}} 

\newtheorem{stheorem}{Theorem}[section]
\newtheorem{sclaim}[stheorem]{Claim}

\newtheorem{scorollary}[stheorem]{Corollary}
\newtheorem{slemma}[stheorem]{Lemma}
\newtheorem{sproposition}[stheorem]{Proposition}

\newtheorem{sremarks}[stheorem]{Remarks}

\theoremstyle{definition}

\numberwithin{equation}{section}

\def\ZZ{\mathbb{Z}}

\def\cA{\mathcal{A}}

\def\cO{\mathcal{O}}

\def\cP{\mathcal{P}}
\def\cT{\mathcal{T}}

\def\etale{\text{\rm \'etale}}

\def\2int{\mathop{2\int}\nolimits}

\def\ua{\underline{a }}

\def\Stab{\mathop{\rm Stab}\nolimits}

\def\Gal{\mathop{\rm Gal}\nolimits}

\def\Aut{\text{\rm{Aut}}}
\def\Out{\text{\rm{Out}}}

\def\resp.{\mathop{\rm resp.}\nolimits}
\def\limproj{\mathop{\oalign{lim\cr
\hidewidth$\longleftarrow$\hidewidth\cr}}}

\def\lgr{\longrightarrow}

\font\math=cmmi10
\def\varpi{\hbox{\math\char'44}}

\def\simlgr{\buildrel\sim\over\lgr}

\def\pa{\S\kern.15em }

\def\un{\uppercase\expandafter{\romannumeral 1}}
\def\deux{\uppercase\expandafter{\romannumeral 2}}
\def\trois{\uppercase\expandafter{\romannumeral 3}}
\def\quatre{\uppercase\expandafter{\romannumeral 4}}
\def\cinq{\uppercase\expandafter{\romannumeral 5}}
\def\six{\uppercase\expandafter{\romannumeral 6}}

\def\hfl#1#2#3{\smash{\mathop{\hbox to#3{\rightarrowfill}}\limits
^{\scriptstyle#1}_{\scriptstyle#2}}}
\def\gfl#1#2#3{\smash{\mathop{\hbox to#3{\leftarrowfill}}\limits
^{\scriptstyle#1}_{\scriptstyle#2}}}

\begin{document}

\title[Tamely-ramified extension]{Semi-simple groups that are quasi-split over a tamely-ramified extension}

\author[P.\ Gille]{Philippe Gille}
\address{Univ Lyon, Universit\'e Claude Bernard Lyon 1, CNRS UMR 5208, Institut Camille Jordan, 43 blvd. du 11 novembre 1918, F-69622 Villeurbanne cedex, France.}
\thanks{The  author is supported by the project ANR Geolie, ANR-15-CE40-0012, (The French National Research Agency).}

\date{\today}

\begin{abstract} 
Let  $K$ be a  discretly henselian field whose residue field is separably closed.
Answering a question raised by G. Prasad, we show that a semisimple $K$--group
$G$ is quasi-split if and only if it quasi--splits after a finite tamely 
ramified extension of $K$.
 
\end{abstract}

\maketitle

\bigskip

\bigskip

\noindent{\bf Keywords:} Linear algebraic groups, Galois cohomology, Bruhat-Tits theory.

\medskip

\noindent{\bf MSC: 20G10, 20G25}.

\bigskip

\section{Introduction} 
Let $K$ be  a discretly valued henselian   field
with valuation ring $\cO$ and residue field $k$. We denote by $K_{nr}$
the maximal unramified extension of $K$ and by $K_t$ its maximal tamely 
ramified extension.
If $G/K$ is a semisimple simply connected groups, 
Bruhat-Tits theory is  available in the sense of \cite{Pr1,Pr2}
and  the Galois cohomology set $H^1(K_{nr}/K,G)$ can be computed 
in terms of the Galois cohomology of special fibers of Bruhat-Tits 
group schemes \cite{BT3}. This permits to compute
$H^1(K,G)$ when the residue field $k$ is perfect.

On the other hand,
if $k$ is not perfect, ``wild cohomology classes'' occur, that is
$H^1(K_{t},G)$ is non-trivial. Such examples appear for example
in the study of bad unipotent elements of semisimple algebraic groups \cite{G2002}.
Under some restrictions on $G$,  we would like to show that  $H^1(K_t/K_{nr},G)$ vanishes
(see Corollary \ref{cor_BT}).
This is related to the following quasi-splitness  result.

\begin{stheorem}\label{main} Let $G$ be a semisimple simply connected 
$K$--group which is quasi-split over $K_t$.

\smallskip

(1) If the residue field $k$ is separably closed, then $G$ is quasi-split.

\smallskip

(2) $G \times_K K_{nr}$ is quasi-split.

\end{stheorem}

This theorem answers a question raised by Gopal Prasad who found 
another proof  by reduction  to the inner case of type $A$ \cite[th. 4.4]{Pr2}.
Our first observation is that the result is quite simple to establish 
under the following additional hypothesis:

 \smallskip

$(*)$ {\it If the variety of Borel subgroups of $G$ carries a $0$-cycle of degree one, 
then it has a $K$-rational point.}

\medskip

Property $(*)$ holds away of $E_8$ (section 2). It is an open question
if $(*)$ holds for groups of type $E_8$. For the $E_8$ case (and actually for a strongly inner $K$--group $G$) of Theorem \ref{main}, our proof is a Galois cohomology argument using Bruhat-Tits buildings (section 3).

We can make at this stage some remarks about the statement. Since $K_{nr}$ is a discretly valued henselian   field
with residue field $k_s$,  we observe  that  (1) implies (2). Also a weak approximation argument \cite[prop. 3.5.2]{GGMB}
reduces  to the   complete case. 
If the residue field $k$ is separably closed of characteristic zero, we have then $\cd(K)= 1$,
so that the result follows from Steinberg's theorem \cite[\S 4.2, cor. 1]{Se1}. In other words, 
the main case to address is that of characteristic exponent  $p> 1$.

\medskip

\noindent{\bf Acknowledgements.}
We are grateful to G. Prasad for raising this interesting question and also for 
fruitful discussions.

\section{The variety of Borel subgroups and $0$--cycle of degree one}

Let $k$ be a field, let $k_s$ be a separable closure and let 
$\Gal(k_s/k)$ be the absolute Galois group of $k$.
Let $q$ be a nonsingular quadratic form.
A celebrated result of Springer states that the Witt index of
$q$ is insensitive to odd degree field extensions. In particular
the property to have a maximal Witt index is insensible to odd degree extensions and 
this can be rephrased by saying that the algebraic group $\SO(q)$
is quasi-split iff it is quasi-split over an odd degree field extension of $k$.
This fact generalizes for all semisimple groups without type $E_8$.

\begin{stheorem}\label{th_springer} Let $G$ be a semisimple 
algebraic $k$-group without quotient of type $E_8$.
Let $k_1, \dots , k_r$ be finite field extensions of $k$ with
coprime degrees. Then 
$G$ is quasi-split if and only if 
$G_{k_i}$ is quasi-split for $i=1,...,r$.
\end{stheorem}

The proof is far to be uniform  hence gathers several 
contributions \cite{BL,Ga2001}. Note that the split version (in the absolutely
almost simple case) is \cite[th. C]{G1997}. 
We remind the reader that a semisimple $k$-group $G$ is 
isomorphic to an inner twist of a quasi-split group $G^q$ and
that such a $G^q$ is unique up to isomorphism.
Denoting by $G^q_{ad}$ the adjoint quotient of $G^q$, 
this means that there exists a Galois 
cocycle $z :\Gal(k_s/k) \to G^q_{ad}(k_s)$
such that $G$ is isomorphic to ${_zG^q}$.
We denote by $\pi: G^{sc,q}\to
G_{ad}^q$ the simply connected cover  of $G^q$.
Then ${_zG^{sc,q}}$ is the  simply connected cover  of ${_zG^q}\cong G$.

\newpage

\begin{slemma}\label{lem_easy} The following are equivalent:
\smallskip

(i)  $G$ is quasi-split;

(ii) $[z] =1 \in H^1(k,G^q_{ad})$;

\smallskip

\noindent If furthermore $[z]= \pi_*[z^{sc}]$ for a $1$-cocycle
$z^{sc}: \Gal(k_s/k) \to G^{sc,q}(k_s)$, (i) and (ii)  also equivalent to 

\smallskip

(iii) $[z^{sc}]=1 \in H^1(k,G^{sc,q})$.

\end{slemma}

\begin{proof} The isomorphism class of $G$ is encoded by the image of
$[z]$ under  the  map $
 H^1(k,G^q_{ad})  \to H^1(k,\Aut(G^q)).$
The right handside map has trivial kernel since
the exact sequence $1 \to G^q_{ad} \to \Aut(G^q) \to 
\Out(G^q) \to 1$ is split (\cite[XXIV.3.10]{SGA3}
or \cite[31.4]{KMRT}), whence the implication $(ii) \Longrightarrow (i)$.
The reverse inclusion $(i) \Longrightarrow (ii)$ is obvious.

Now we assume that $z$ lifts to a $1$-cocycle $z^{sc}$.
The implication $(iii) \Longrightarrow (ii)$ is then  obvious.
The point is that the map $H^1(k,G^{sc,q}) \to H^1(k,G^{q})$ 
 has trivial kernel \cite[III.2.6]{G1997}
whence the implication $(ii) \Longrightarrow (iii)$.
\end{proof}

 We proceed to the proof of Theorem \ref{th_springer}.

\begin{proof} 
Let $X$ be the variety of Borel subgroups of $G$ \cite[XXII.5.8.3]{SGA3}, 
a projective $k$--variety. The $k$--group $G$ is quasi-split iff $X$ has a $k$-rational point.
Thus we have to prove that if $X$ has a $0$-cycle of degree one,
then $X$ has a $k$-point.

Without loss of generality, we can assume that 
$G$ is simply connected. According to \cite[XXIV.5]{SGA3} we have that $G \simlgr \prod_{j=1,..,s} R_{l_j/k}(G_j)$
where $G_j$ is an absolutely almost simple simply connected group defined over 
a finite separable field extension $l_j$ of $k$ (the notation $R_{l_j/k}(G_j)$ stands as usual for the
Weil restriction to $k_k$ to $k$).
The variety of Borel subgroup $X$ of $G$ is then isomorphic
to $\prod_{j=1,..,s} R_{l_j/k}(X_j)$
where $X_j$ is the $l_j$-variety of Borel subgroups of $G_j$.

\smallskip

\noindent{\it Reduction to the absolutely almost simple case.}
Our assumption is that  $X(k_i) \not = \emptyset$ for $i=1,..,r$
hence $X_j(k_i \otimes l_j) \not = \emptyset$ for $i=1,..,r$ and $j=1,..,s$.
Since $l_j/k$ is separable, $k_i \otimes l_j$ is an \'etale $l_j$-algebra
for $i=1,..,r$ and it follows that $X_j$ carries a $0$-cycle of degree one.
If we know to prove the case of each $X_j$, we have $X_j(k_j) \not = \emptyset$
hence $X(k) \not = \emptyset$.
From now on, we assume that $G$ is absolutely almost simple.
We denote by $G_0$ the Chevalley group over $\ZZ$
such that $G$ is a twisted form of $G_0 \times_{\ZZ} k$.

\smallskip

\noindent{\it Reduction to the characteristic zero case.}
If $k$ is of  characteristic $p>0$, 
let $\cO$ be a Cohen ring for the residue field $k$, that is 
a complete discrete valuation ring such that its fraction field
$K$ is of characteristic zero and for which $p$ is an uniformizing parameter 
\cite[IX.41]{BAC78}.
The isomorphism class of $G$ is encoded by a Galois cohomology
class in $H^1(k, \Aut(G_0))$. Since $\Aut(G_0)$ is a smooth affine
$\ZZ$-group scheme \cite[XXIV.1.3]{SGA3}, we can use  Hensel's lemma
$H^1_{\etale}(\cO, \Aut(G_0)) \simlgr H^1(k, \Aut(G_0))$ \cite[XXIV.8.1]{SGA3}
so that $G$ lifts in a  semisimple simply connected group scheme $\gG$ over $\cO$.
Let $\gX$ be the $\cO$--scheme of Borel subgroups of $\gG$ \cite[XXII.5.8.3]{SGA3}. It 
is smooth and projective. 
For $i=1,..,r$, let $K_i$ be an unramified field  extension of
$K$ of degree $[k_i:k]$ and of residue field $k_i$.
Denoting by $\cO_i$ its valuation ring, we consider the 
maps
$$
\gX(K_i) = \gX(\cO_i) \to \hskip-3mm \to X(k_i).
$$
The left equality come from the projectivity and the right
surjectivity is Hensel's lemma. It follows that 
$\gX(K_i)\not = \emptyset$
for $i=1,...,r$ so that $\gX_K$ has a $0$-cycle of degree one.
Assuming the result in the characteristic zero case, it follows that
$\gX(K) = \gX(O) \not=\emptyset$ whence $X(k) \not=\emptyset$.
We may assume from now that  $k$ is of characteristic zero.
We denote by $\mu$ the center of $G$ and by $t_G \in H^2(k, \mu)$
the Tits class of $G$ \cite[\S 31]{KMRT}.
Since the Tits class of the quasi-split form $G^q$ of $G$ is zero, 
 the classical  restriction-corestriction argument yields
 that $t_G=0$. In other words
 $G$ is a strong inner form of its quasi-split form $G^q$.
It means that there exists a Galois cocycle $z$ with value in
$G^q(k_s)$ such that $G \cong {_zG^q}$, that is
the twist by inner conjugation of $G$ by $z$. 
Lemma \ref{lem_easy} shows that our problem
is rephrased  in Serre's question \cite[\S 2.4]{Se2} on the triviality
of the kernel of the map
$$
H^1(k,G^q) \to \prod\limits_{i=1,...,r} H^1(k_i,G^q)
$$
That kernel is indeed trivial in our case \cite[Th. 0.4]{B}, whence the result.
\end{proof}

We remind  the reader that one can associate to a semisimple $k$-group
$G$ its set $S(G)$ of torsion primes which depends only of its Cartan-Killing
type \cite[\S 2.2]{Se2}.
Since an algebraic group splits after an extension
of degree whose primary 
factors belong to $S(G)$ \cite{T1992}, we get  the following refinement.

\begin{scorollary}\label{cor_springer} Let $G$ be a semisimple 
algebraic $k$-group without quotient of type $E_8$.
Let $k_1, \dots , k_r$ be finite field extensions of $k$ such that 
$\mathrm{g.c.d.}( [k_1:k], \dots, [k_r:k])$ is prime to $S(G)$.
 Then  $G$ is quasi-split if and only if 
$G_{k_i}$ is quasi-split for $i=1,...,r$.
\end{scorollary}

Lemma  \ref{lem_easy} together with the Corollary  implies
the following statement.

\begin{scorollary}\label{cor_springer2} Let $G$ be a semisimple
simply connected quasi-split algebraic $k$-group without factors
of type $E_8$.
Let $k_1, \dots , k_r$ be finite field extensions of $k$ such that 
$\mathrm{g.c.d.}( [k_1:k], \dots, [k_r:k])$ is prime to $S(G)$.
Then the maps
$$
H^1(k,G) \to \prod\limits_{i=1,...,r} H^1(k_i,G) \quad \hbox{and}
$$
$$
H^1(k,G_{ad}) \to \prod\limits_{i=1,...,r} H^1(k_i,G_{ad})
$$
have trivial kernels. \qed
\end{scorollary}

We can proceed now on the proof of Theorem \ref{main}.(1)  away of $E_8$
since Theorem \ref{th_springer} shows that the condition $(*)$
is fullfilled in that case.

\medskip

\noindent{\it Proof of Theorem \ref{main}.(1) under  assumption $(*)$.}
Here $K$  is   a discretly valued henselian  field.
We are given a semisimple  $K$--group $G$ satisfying  assumption $(*)$, and such
that $G$ becomes quasi-split after a finite tamely ramified extension
$L/K$. Note that $[L:K]$ is prime to $p$.
We denote by $X$ the $K$--variety of Borel subgroups of $G$.
We want to show that $X(K) \not = \emptyset$.  We are then reduced to the following cases:
 
 \smallskip
 
 (i) $K$ is perfect and the absolute Galois group $\Gal(K_s/K)$ is a pro-$l$-group for a prime
 $l \not = p$.

 \smallskip
 
 (ii) $\Gal(K_s/K)$ is a pro-$p$-group.
 
 \smallskip
 
 \noindent By weak approximation \cite[prop. 3.5.2]{GGMB}, we may assume that 
 $K$ is complete.  Note that this operation does not change the 
 absolute Galois group ({\it ibid}, 3.5.1).
 
 \smallskip

 \noindent{\it Case (i):}
 We have that $\cd_l(K) \leq \cd_l(k)+1= 1$ 
 \cite[\S II.4.3]{Se1} so that $\cd(K) \leq 1$.
 Since $K$ is perfect, Steinberg's theorem \cite[\S 4.2, cor. 1]{Se1}
 yields that $G$ is quasi-split.

\smallskip

\noindent{\it Case (ii):} The extension $K$ has no proper tamely ramified
extension hence our assumption implies that $G$ is quasi-split.
\qed

\medskip

\begin{sremarks} {\rm 
\noindent $a)$ In case (i) of the proof, there is no 
 need to assume that $K$ is perfect and $l$  can be any
prime different from $p$. The point is that if $Gal(K_s/K)$ is a pro-$l$-group,
then the separable cohomological dimension of $K$ is less than or
equal to $1$, and then any semi-simple $K$-group is quasi-split, see
\cite[\S 1.7]{Pr1}

\smallskip

\noindent  $b)$ It an open question  whether a $k$--group of type $E_8$ is split 
if it is  split after coprime degree extensions $k_i/k$. 
A positive answer to this  question would imply Serre's vanishing conjecture II for 
groups of type $E_8$ \cite[\S 9.2]{G2017}.

\smallskip

\noindent  $c)$ Serre's injectivity question has a positive answer for an arbitrary classical group
(simply connected or adjoint)
 and holds for certain exceptional cases \cite{B}.
}
\end{sremarks}

\section{Cohomology and buildings}

The field $K$ is as in the introduction.

\begin{sproposition} \label{prop_BT} Assume that $k$ is separably closed.
Let $G$ be a split semisimple connected $K$-group.
 Then $H^1(K_t/K,G)=1$.
\end{sproposition}

\begin{proof} 
We can reason at finite level and shall prove that 
$H^1(L/K,G)=1$ for a given 
  finite tamely ramified extension of $L/K$. We put $\Gamma=\Gal(L/K)$, it is
a cyclic group whose  order $n$ is  prime to the characteristic exponent $p$ of
$k$.

Let $\calB(G_L)$ be the Bruhat-Tits building of 
$G_L$.  It comes equipped with an action of $G(L) \rtimes \Gamma$ \cite[\S 4.2.12]{BT2}.
Let $(B,T)$ be a Killing couple for  $G$.
The split $K$--torus  $T$ defines an apartment $\cA(T_L)$
of $\calB(G_L)$ which is preserved by the action of 
$N_G(T)(L) \rtimes \Gamma$.

 We are given  a Galois cocycle $z :\Gamma \to G(L)$; it defines
 a section $u_z: \Gamma \to G(L) \rtimes \Gamma, \sigma \mapsto z_\sigma 
 \sigma $ of the projection map $G(L) \rtimes \Gamma \to\Gamma$.
 This provides an action of $\Gamma$ on $\calB(G_L)$ called the twisted action 
with respect to the cocycle $z$. The Bruhat-Tits fixed point theorem \cite[\S 3.2]{BT1} provides a point
 $y \in \calB(G_L)$ which is fixed by the twisted action.
 This point belongs to an apartment and since $G(L)$ acts transitively on the
 set of apartments of $\calB(G_L)$ there exists a suitable $g \in G(L)$
 such that $g^{-1}.y=x \in \cA(T_L)$. We observe that 
 $\cA(T_L)$ is fixed pointwise by $\Gamma$ (for the standard action), so that 
$x$ is fixed under $\Gamma$. We consider the equivalent cocycle
$z'_\sigma= g^{-1} \, z_\sigma \, \sigma(g)$ and compute

\begin{eqnarray} \nonumber
 z'_\sigma \, . \, x& = & z'_\sigma \, . \, \sigma(x) \\ \nonumber
 & = &(g^{-1} \, z_\sigma \, \sigma(g)) (\sigma(g^{-1}).\sigma(y) )\\ \nonumber
 & = & g^{-1} \, . \, \bigl( ( z_\sigma \sigma).y \bigr) \\ \nonumber
 &=&  g^{-1} \, . \, y  \qquad 
 \hbox{[$y$ is fixed under the twisted action]} \\ \nonumber
 &=& x .
\end{eqnarray}

\noindent Without loss of generality, we may assume that $z_\sigma.x=x$ for 
each $\sigma \in \Gamma$. We put $P_x= \Stab_{G(L)}(x)$;
since $x$ is fixed by $\Gamma$, the group $P_x$ is preserved by the 
action of $\Gamma$. Let $\cP_x$ the Bruhat-Tits $\cO_L$-group 
scheme attached to $x$. We have  $\cP_x(\cO_L)= P_x$  and 
we know that its special fiber $\cP_x \times_{\cO_L} k$
is smooth  connected, that its quotient $M_x= (\cP_x \times_{\cO_L} k)/ U_x$
by its split unipotent radical $U_x$ is split reductive.

An important  point is that the action of $\Gamma$ on $\cP_x(O_L)$
arises from a semilinear action of $\Gamma$ on the $O_L$--scheme $\cP_x$ as explained 
in the beginning of \S 2 of  \cite{Pr2}. It induces then a $k$--action 
of the group $\Gamma$ on $\cP_x \times_{\cO_L} k$, on $U_x$ and on $M_x$.
 Since $x$ belongs to $\cA(T_L)$, $\cP_x$ carries a natural maximal split $\cO_L$--torus $\cT_x$
and $T_x= \cT_x \times_{O_L} k$ is a maximal $k$--split torus
of $\cP_x \times_{O_L} k$ and its image in 
$M_x$ still denoted  by $T_x$ is a maximal $k$-split torus of
$M_x$. We observe that $\Gamma$ acts 
trivially on the $k$-torus $T_x$. But $T_{x}/C(M_x)= \Aut(M_x, id_{T_x})$
\cite[XXIV.2.11]{SGA3}, it follows that $\Gamma$ acts on $M_x$ 
by means of a group homomorphism $\phi: \Gamma \to T_{x,ad}(k)$
where $T_{x,ad}= T_x /C(M_x) \subseteq M_x/C(M_x)=M_{x,ad}$.
For each $m \in M_x(k)$, we have $\sigma(m)= \mathrm{int}(\phi(\sigma)). m$.

Now we take a generator $\sigma$ of $\Gamma$ and denote by
$a_\sigma$ the image in $M_x(k)$ of $z_\sigma \in P_x$
and by $\underline{a}_\sigma$ its image in $(M_x/C(M_x))(k)$.
The cocycle relation yields  $\ua_{\sigma^2}=\ua_\sigma \sigma(\ua_\sigma)=
\ua_\sigma \phi(\sigma)\ua_\sigma \phi(\sigma)^{-1}$
and more generally (observe that $\phi(\sigma)$ is fixed by $\Gamma$)
$$
\ua_{\sigma^j}=
\ua_\sigma \phi(\sigma)\ua_\sigma \phi(\sigma)^{-1}
\dots 
\dots \phi(\sigma)^{j-1} \ua_\sigma \phi(\sigma)^{1-j} \,
\phi(\sigma)^j \ua_\sigma \phi(\sigma)^{-j}
= \bigl(\ua_\sigma \phi(\sigma) \bigr)^j \, \phi(\sigma)^{-j}
$$
for $j=2,..,n$. Since $\phi(\sigma)^{n}=1$,  we get  the relation
$$
1= (\ua_\sigma \phi(\sigma))^n.
$$
Then $\underline{a}_\sigma \phi(\sigma)$ is an element of
order $n$ of $M_{x,ad}(k)$ so is semisimple.
But $k$ is separably closed so that $\underline{a}_\sigma \phi(\sigma)$
belongs to a maximal $k$-split torus $^m T_{x,ad}$ with $m \in M_x(k)$.
It follows that $m^{-1} \underline{a}_\sigma \phi(\sigma) m \in T_{ad,x}(k)$.
Since $\phi(\sigma)$ belongs to $T_{ad,x}(k)$, we have that
$m^{-1} \underline{a}_\sigma \phi(\sigma) m \phi(\sigma)^{-1} \in  T_{ad,x}(k)$
hence $m^{-1} \underline{a}_\sigma \sigma(m) \in T_{ad,x}(k)$.
It follows that   $m^{-1} a_\sigma \sigma(m) \in T_{x}(k)$.
Since the map $\cP_x(O_L) \to M_x(k)$ is surjective
 we can then assume that $a_\sigma \in T_x(k)$ without loss of generality
 so that  the cocycle $a$ takes value in $T_x(k)$. But $T_x(k)$ is a trivial $\Gamma$-module
so that $a$ is given by a homomorphism $f_a: \Gamma \to T_x(k)$.
This homomorphism lifts (uniquely) to a homomorphism $\widetilde f_a: \Gamma
\to \cT_x(\cO_L)^\Gamma$. The main technical step is 

\begin{sclaim} The fiber of  $H^1(\Gamma, P_x) \to H^1(\Gamma, M_x(k))$ 
at $[f_a]$ is $\bigl\{ [\widetilde f_a] \bigr\}$.
\end{sclaim}

\noindent Using the Claim, we have $[z]=[\widetilde f_a] \in H^1(\Gamma, P_x)$.
Its image in $H^1(\Gamma, G(L))$ belongs to the image of the map
$H^1(\Gamma, \cT_x(L)) \to H^1(\Gamma, G(L))$. But 
$0= H^1(\Gamma, \cT_x(L))$ (Hilbert 90 theorem) thus $[z]=1 \in 
H^1(\Gamma, G(L))$ as desired.

It remains to establish the Claim.
We put $P_x^\star=\ker(P_x \to M_x(k))$ and this group  
can be filtered by a $\Gamma$-stable decreasing filtration 
by normal subgroups ${\calU^{(i)}}_{i \geq 0}$ such that for each $i\leq j$
there is a split unipotent $k$-group 
$U^{(i,j)}$ equipped with an action of $\Gamma$
such that $\calU^{(i)}/ \calU^{(j)}=U^{(i,j)}(k)$
\cite[page 6]{Pr2}. 
We denote by ${_{\widetilde f_a}P_x}^\star$ the $\Gamma$--group
${P_x}^\star$ twisted by the cocycle $\widetilde f_a$;
there is a surjection $H^1(\Gamma, {_{\widetilde f_a}P_x}^\star)$
on the fiber at $[ f_a]$ of the map $H^1(\Gamma, P_x) \to H^1(\Gamma, M_x(k))$ \cite[I.5.5, cor. 2]{Se1}.
It is then enough to show that $H^1(\Gamma, {_{\widetilde f_a}P_x}^\star)=1$.
It happens fortunately that the filtration is stable under
the adjoint action of the image of $\widetilde f_a$.
By using the pro-unipotent $k$-group $U=\limproj U^{(0,j)}$
and  Lemma \ref{lem_unipotent} in the next subsection, we have that
$H^1(\Gamma, {_{\widetilde f_a}P_x}^\star)=H^1\bigl(\Gamma, ({_{\widetilde f_a}U})(k) \bigr)=1$.
Since $H^1\bigl(\Gamma, ({_{\widetilde f_a}U})(k) \bigr)$ maps onto the kernel of
fiber of  $H^1(\Gamma, P_x) \to H^1(\Gamma, M_x(k))$
at $[\widetilde f_a]$ \cite[\S I.5.5, cor. 2]{Se1},  we conclude that the Claim is established.
\end{proof}

\smallskip

This permits to complete the proof of Theorem \ref{main}.

\medskip

\noindent{\it Proof of Theorem \ref{main}.(1)}. 
By the usual reductions, the question boils down to 
the semisimple simply connected case and even the  absolutely almost $K$--simple semisimple simply connected case.
Taking into account the cases established in   section 2, 
it remains to deal with the case of type $E_8$. Denote by $G_0$ the split group of type $E_8$, 
 we have $G_0= \Aut(G_0)$. It follows that $G \cong {_zG_0}$
 with $[z] \in H^1(K,G_0)$. Our assumption is that 
 $G_{K_{t}}$ is quasi-split so that $[z] \in H^1(K_{t}/K,G_0)$.
 Proposition \ref{prop_BT} states that $H^1(K_{t}/K,G_0)=1$, whence $G$ is split.
 \hfill\hfill $\square$

\medskip

We record the following cohomological application.

\begin{scorollary}\label{cor_BT}
 Let $G$ be a  semisimple  algebraic $K$--group which is quasi-split over $K_{t}$. 
 We assume that $G$ is simply connected or adjoint.
 Then $H^1(K_t /K_{nr},G)=1$.
\end{scorollary}

\begin{proof} Theorem \ref{main} permits to assume that $G$ is quasi-split. 
We denote by $\pi: G \to G_{ad}$ the adjoint quotient of $G$.
Since the map $H^1(K,G) \to  H^1(K,G_{ad})$ has trivial kernel \cite[lemme III.2.6]{G1997},
we can assume that $G$ is adjoint. 
Let $[z] \in H^1(K_t/K_{nr},G)$. We consider the twisted $K_{nr}$--form $G'= {_zG}$ of $G$.
Since $G'_{K_t}$ is isomorphic to $G_{K_t}$, $G'_{K_t}$ is quasi-split and Theorem
\ref{main} shows that $G'$ is quasi-split hence isomorphic to $G$. It means that $z$ belongs to the 
kernel of  the map $\mathrm{int}_*: H^1(K,G) \to H^1(K,\Aut(G))$.
But the exact sequence of $K$--groups $1 \to G \xrightarrow{\mathrm{int}} \Aut(G) \to \Out(G) \to 1$ 
splits \cite[XXIV.3.10]{SGA3} so that the above kernel is trivial. Thus $[z]=1 \in  H^1(K_{nr},G)$.
\end{proof}

\section{Appendix: Galois cohomology of pro-unipotent groups}

Let $k$ be a separably closed  field.
Let $U$ be a pro-unipotent algebraic $k$-group equipped with 
an action of a finite group $\Gamma$, that is $U$ admits a decreasing filtration
$U=U_0 \supset U_1 \supset U_2 \supset \cdots$ by normal
pro unipotent $k$--groups which are stabilized by $\Gamma$ and such that
$U_i/U_{i+1}$ is an unipotent algebraic $k$-group for $i=1,...,n$.

\begin{slemma}\label{lem_unipotent}
 We assume that $\sharp \Gamma$ is invertible in $k$ and that
 the $U_i/U_{i+1}$'s are smooth and connected. Then $H^1(\Gamma, U(k))=1$.
\end{slemma}

\begin{proof}
We start with the algebraic case, that is
of a smooth connected unipotent $k$--group.
According to \cite[XVII.4.11]{SGA3}, $U$ admits a central characteristic filtration 
$U=U_0 \supset U_1 \supset \dots  \supset U_n=1$
such that $U_i/U_{i+1}$ is a twisted form of a $k$--group $\GG_a^{n_i}$.
Since $U_{i+1}$ is smooth and $k$ is separably closed,
we  have   the  following exact sequence of $\Gamma$--groups
$$
1 \to U_{i+1}(k) \to U_{i}(k)   \to (U_i/U_{i+1})(k) \to 1.
$$
The multiplication by $\sharp \Gamma$ on the abelian  group $(U_i/U_{i+1})(k)$ is 
an isomorphism so that   $H^1(\Gamma, (U_i/U_{i+1})(k))=0$. The exact sequence above
shows that  the map $H^1(\Gamma, U_{i+1}(k)) \to H^1(\Gamma, U_i(k))$ is onto.
By induction it follows that $1=H^1(\Gamma, U_n(k))$ maps onto
$H^1(\Gamma, U(k))$ whence $H^1(\Gamma, U(k))=1$.

We consider now the pro-unipotent case. 
Since the $U/U_i$'s are smooth, we have that $U(k)=\limproj (U/U_i)(k)$.
Therefore by successive approximations the kernel of the map 
$$
H^1(\Gamma, U(k)) \to \limproj H^1\bigl(\Gamma, (U/U_i)(k) \bigr)
$$
is trivial. But according to the first case, the right handside is trivial thus
$H^1(\Gamma, U(k))=1$.
\end{proof}

\addvspace{\bigskipamount}

\bigskip


\begin{thebibliography}{EGA4}

\bibitem{BL}  E. Bayer-Fluckiger, H.W. Lenstra Jr., 
{\it Forms in odd degree extensions and 
self-dual normal bases}, Amer. J. Math. {\bf 112} (1990), 359--373.
 
 
\bibitem{B} J. Black, {\it Zero cycles of degree one on 
  principal homogeneous spaces}, J. Algebra {\bf 334} (2011), 232-246.


\bibitem{BAC78} N.\,Bourbaki,
{\em Alg\`ebre commutative} (Ch. 7--8), Springer--Verlag, Berlin, 2006.
  
  
\bibitem{BT1} F. Bruhat, J. Tits, {\it Groupes r\'eductifs sur un corps local. I},
 Inst. Hautes Etudes Sci. Publ. Math.  {\bf  41}  (1972), 5--251.


\bibitem{BT2} F. Bruhat, J. Tits, {\it Groupes alg\'ebriques sur un
corps local II. Existence d'une donn\'ee radicielle valu\'ee}, Pub. Math. IHES  {\bf 60}  (1984),  5--184.


\bibitem{BT3} F. Bruhat, J. Tits, {\it Groupes alg\'ebriques sur un
corps local III. Compl\'ements et application \`a la cohomologie
galoisienne}, J. Fac. Sci. Univ. Tokyo  {\bf 34}  (1987),   671--698.






\bibitem{GGMB} O. Gabber, P. Gille,  L. Moret-Bailly,
{\it Fibr\'es principaux sur les corps hens\'eliens},
Algebraic Geometry {\bf 5} (2014), 573-612. 


\bibitem{Ga2001} S. Garibaldi, {\it The Rost invariant has trivial kernel for quasi-split groups of low rank},  
Comment. Math. Helv. {\bf   76}  (2001),  684--711.


\bibitem{G1997} P. Gille, {\it
La R-\'equivalence sur les groupes  alg\'ebriques r\'eductifs d\'efinis sur un corps global}, 
Publications  Math\'ematiques de l'I.H.\'E.S. {\bf  86} (1997), 199-235.


\bibitem{G2002} P. Gille, {\it Unipotent subgroups of 
reductive groups of  characteristic  p>0}, 
Duke Math. J. {\bf 114} (2002), 307-328.


\bibitem{G2017} P. Gille, {\it Groupes alg\'ebriques semi-simples sur un 
corps de dimension cohomologique s\'eparable $\leq 2$}, 
monograph in preparation.


\bibitem{KMRT} M-A.\,Knus, A.\,Merkurjev, M.\,Rost,  J-P.\,Tignol, {\em The book of involutions},
AMS Colloq. Publ. {\bf 44}, Providence, 1998.

\bibitem{Pr1} G. Prasad, {\it A new approach to unramified
descent in Bruhat-Tits theory}, preprint (2016),  arXiv:1611.07430.


\bibitem{Pr2} G. Prasad, {\it Finite group actions on reductive
groups and tamely-ramified descent in Bruhat-Tits theory}, preprint (2017), arXiv:1705.02906.

\bibitem{SGA3} {\it S\'eminaire de G\'eom\'etrie alg\'ebrique de
l'I.H.E.S., 1963-1964, sch\'emas en groupes, dirig\'e par M.
Demazure et A. Grothendieck},  Lecture Notes in Math. 151-153.
Springer (1970).

\bibitem{Se1} J-P.\,Serre, {\em Cohomologie galoisienne}, cinqui\`eme  \'edition, 
Springer-Verlag, New York, 1997.

\bibitem{Se2} J-P.\,Serre, {\it Cohomologie galoisienne: Progr\`es et probl\`emes}, S\'eminaire Bourbaki, expos\'e 783 
(1993-94), Ast\'erisque  {\bf 227} (1995).




\bibitem{T1992} J. Tits, {\it Sur les degr\'es des extensions de corps d\'eployant les groupes
alg\'ebriques simples}, C. R. Acad. Sci. Paris S{\'e}r. I Math. \textbf{315}
(1992),  1131--1138.
 
 
 \end{thebibliography}
\end{document}